\documentclass[11pt,reqno]{amsart}
\usepackage{amssymb,latexsym}
\usepackage{amsmath,color}
\usepackage{amsthm}
\usepackage{epsfig}
\usepackage{graphicx}
\usepackage{fancybox,framed}
\usepackage{titletoc}
\usepackage{dsfont}
\usepackage{nccbbb}
\usepackage{amsmath}
\usepackage{amssymb}
\usepackage{amssymb}
\usepackage{amstext}
\usepackage{amsopn}
\usepackage{amsbsy}
\usepackage{amscd}
\usepackage{amsxtra}
\usepackage{accents}
\usepackage{bbm}
\usepackage{yfonts}
\usepackage{bbm}

\numberwithin{equation}{section}

\usepackage{xparse}
\usepackage{microtype}
\usepackage[margin=1cm]{geometry}

\usepackage{lipsum}

\ExplSyntaxOn
\box_new:N \l_hideparbox_box

\NewDocumentCommand{\hideparbox}{O{c}mm+m}
 {
  \group_begin:
  \vbox_set:Nn \l_hideparbox_box
   {
    \use:c { @parboxrestore }
    \hsize=#3\scan_stop:
    \strut#4\par
   }
  \vbadness=\c_ten_thousand 
  \vbox_set_split_to_ht:NNn \l_hideparbox_box \l_hideparbox_box { #2 }
  \parbox[#1][#2]{#3}
   {
    \vbox_unpack:N \l_hideparbox_box
   }
  \group_end:
 }
\ExplSyntaxOff

\usepackage[T1]{fontenc}
\newtheorem{theorem}{Theorem}[section]
\newtheorem{proposition}[theorem]{Proposition}

\newtheorem{corollary}[theorem]{Corollary}
\newtheorem{remark}{Remark}[section]

\theoremstyle{definition}

\newtheorem{example}{Example}[section]

\newcommand{\ds}{\displaystyle}

\newcommand{\dz}{\mathrm{d}z}

\newcommand{\bi}{b_{\mathrm{i}}}
\newcommand{\bo}{b_{\mathrm{o}}}
\newcommand{\phl}{\phi_{\lambda}}
\newcommand{\psl}{\psi_{\lambda}}
\newcommand{\gi}{\Gamma_{\lambda,\mathrm{i}}}
\newcommand{\go}{\Gamma_{\lambda,\mathrm{o}}}
\newcommand{\gfi}{g_{\mathrm{i}}}
\newcommand{\gfo}{g_{\mathrm{o}}}
\newcommand{\ue}{u_\lambda}
\newcommand{\abar}{\vec{\boldsymbol{\mathrm{a}}}}
\newcommand{\lomega}{\Omega_{\lambda}}

\newcommand{\ve}{v_{\lambda}}
\newcommand{\dx}{\,\mathrm{d}x}
\newcommand{\dy}{\,\mathrm{d}y}

\usepackage[misc]{ifsym}
\allowdisplaybreaks

\title[Linear convection--diffusion equations with integral boundary conditions]{\bf On the uniqueness of linear convection--diffusion equations with integral boundary conditions}

\address{ \tt (Chiun-Chang Lee) \rm Institute for Computational and Modeling Science, National Tsing Hua University, Hsinchu 30013, Taiwan}
\email{chiunchang@gmail.com, lee2@mx.nthu.edu.tw}

\address{\tt (Masashi Mizuno)  \rm Department of Mathematics, College of Science and Technology, Nihon University, 1-8-14 Kanda-Surugadai, Chiyoda-Ku, Tokyo, 101-8308, Japan}
\email{mizuno.masashi@nihon-u.ac.jp}

\address{\tt (Sang-Hyuck Moon)  \rm Department of Mathematical Sciences, College of Natural Sciences, Ulsan National Institute of Science and Technology, Republic of Korea}
\email{mshyeog@gmail.com}

\begin{document}
{\scriptsize 
-\vspace{-16mm} \\
\hspace*{6pt}Accepted by C. R. Math. Acad. Sci. Paris\hspace{125pt}
Revised manuscript submitted (1st 2022/4/30; 2nd 2022/6/6)
}
\vspace{-3.6mm} \\
\rule{48em}{0.2pt}\\

\maketitle
\vspace*{-8pt}
\begin{center}
\small  Chiun-Chang Lee\footnote[100]{Author to whom the correspondence should be addressed.}  \ \ \ \ \ \  Masashi Mizuno  \ \ \ \ \ \ Sang-Hyuck Moon
\end{center}
\begin{center}
    ...................................................................................................................................................\vspace{-8pt}
\\
\end{center}
\begin{abstract} 
{\footnotesize 

This work contributes to an understanding of the domain size's effect on the existence and uniqueness of the linear convection--diffusion equation with integral-type boundary~conditions, where boundary conditions depend non-locally on unknown solutions. Generally, the uniqueness result of this type of equation is unclear. In this preliminary study, a uniqueness result is verified when the domain is sufficiently large or small. The main approach has an advantage of transforming the integral boundary conditions into new Dirichlet boundary conditions so that we can obtain refined estimates, and the comparison theorem can be applied to the equations. Furthermore, we show a domain such that under different boundary data, the equation in this domain can have infinitely numerous solutions or no solution. 
\vspace{1pt}
\\ 
{\sc Keywords.} {\tt\tiny Convection--diffusion equations $\cdot$ Integral boundary conditions $\cdot$ Expanding domains $\cdot$ Capacity $\cdot$ Uniqueness}\\
{\sc AMS  Classification.}  {\tt\tiny 34B10 $\cdot$ 34D15 $\cdot$  34E05 $\cdot$  34K26 $\cdot$ 35J25}}\vspace{-16pt}\\
\end{abstract}


\pagenumbering{arabic}

\begin{center}
     ...................................................................................................................................................
\end{center}

\section{\bf Introduction and the statement of the main results} 
\noindent

This study examines the existence and uniqueness of the solutions to linear convection--diffusion equations with integral boundary conditions, where the domain has numerous boundary components and a positive parameter.

We first consider a domain with two boundary components for clarity.  Define a smooth bounded domain
\begin{align}\label{domain}
   \lomega:=\left\{\lambda z:z\in\Omega:=\Omega_{\mathrm{out}}\setminus\overline{\Omega_{\mathrm{in}}}\right\}
\end{align}
which is diffeomorphic to an annulus and expands as $\lambda$ tends to infinity, where $\Omega_{\mathrm{in}}$ and $\Omega_{\mathrm{out}}$ represent bounded simply-connected smooth domains in $\mathbb{R}^N$, $N\geq2$, such that $0\in\Omega_{\mathrm{in}}\subsetneqq\Omega_{\mathrm{out}}$.   This study focuses primarily on the linear convection--diffusion equation
\begin{align}\label{equl}
-\Delta\ue(x)+\abar(x)\cdot\nabla\ue(x)+h(x)\ue(x)=0,\quad\,x\in\lomega,
\end{align}
subject to the following integral-type boundary conditions for $\ue$:
\begin{align}\label{bdu}
\begin{cases}
\ds\ue(x)=\bi+  \int_{\lomega}\gfi(y)\ue(y)\dy,\quad\,\,x\in\gi:=\left\{\lambda z:z\in\partial\Omega_{\mathrm{in}}\right\},\vspace{3pt}\\
\ds\ue(x)=\bo+  \int_{\lomega}\gfo(y)\ue(y)\dy,\quad\,x\in\go:=\left\{\lambda z:z\in\partial\Omega_{\mathrm{out}}\right\}.
\end{cases}
\end{align}
The boundary data of $\ue$ on $\gi\cup\go$ satisfies an implicit form with non-local dependence on the unknown $\ue$ in  $\lomega$. Here $\bi$ and $\bo$ are constants, which are independent of $\lambda$, $\gfi,\,\gfo\in\mathrm{C}(\mathbb{R}^N\setminus\{0\};\mathbb{R})$, and the variable coefficients  {$\abar\in\mathrm{C}^{\alpha}(\mathbb{R}^N\setminus\{0\};\mathbb{R}^N)$ and $h\in\mathrm{C}^{\alpha}(\mathbb{R}^N\setminus\{0\};\mathbb{R})$ are H\"{o}lder continuous with exponent~$\alpha\in(0,1)$.} Thus, the restrictions  $\gfi\big|_{\overline{\lomega}}$, $\gfo\big|_{\overline{\lomega}}$, $\abar\big|_{\overline{\lomega}}$ and $h\big|_{\overline{\lomega}}$ are well defined for any $\lambda>0$. The conditions of these functions will be assumed specifically {in \eqref{ab619} and \eqref{gio}}.

Previous research literature \cite{C2005,F2011,L2013,SD-2007}, in which a one-dimensional (1D) case was numerically examined, inspired the equation that we are studying. We refer the readers to \cite{A1986,CA2021} for the details on the relevant qualitative and asymptotic analysis for solutions of singularly perturbed models with different integral-type boundary conditions. Furthermore, several studies have investigated the uniqueness or multiplicity of the solutions to such non-local equations (e.g. see [4] and its references). The primary approach is based on several fixed-point theorems, such as Krasnosel'skii's fixed-point theorem, Schauder fixed-point theorem and the weakly contractive mapping theorem. Different from 1D models, for fixed $\lambda>0$, studying the uniqueness or multiplicity of the solutions of higher-dimensional equations such as (1.2)--(1.3) is a challenge in the analysis because such equations do not have a variational structure.

\begin{figure}[ht]
\includegraphics[width=10.8cm]{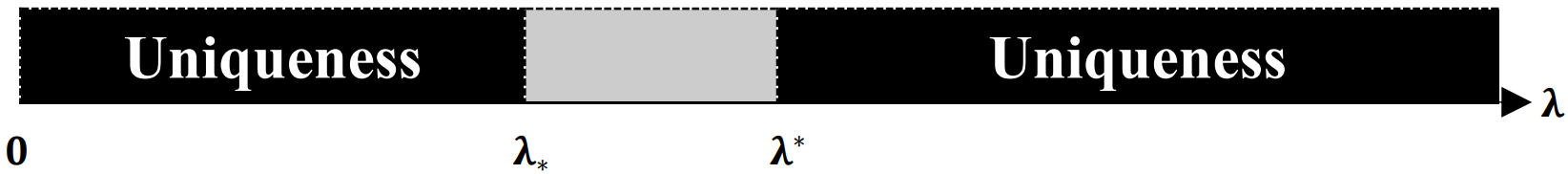}
\caption{The existence, uniqueness and multiplicity of solutions to equation~\eqref{equl}--\eqref{bdu} may be affected by the domain $\lomega$ with respect to the parameter $\lambda$.  In Theorem~\ref{m-thm}, we demonstrate that there exist $0<\lambda_*<\lambda^*$ such that equations~\eqref{equl}--\eqref{bdu} have a unique solution as $\lambda\in(0,\lambda_*)\cup(\lambda^*,\infty)$. However, when $\lambda\in[\lambda_*,\lambda^*]$ (the grey region), the equation may have multiple solutions or no solution. (For clarity, we refer the reader to  Example~\ref{ex1} and Proposition~\ref{m-cor}.)
}\label{fig1}
\end{figure}

 $\lomega$ with $\lambda\to\infty$ is an expanding domain, which is related to the mathematical problems for convection--diffusion models in a reactor of macroscopic length scale (see, e.g. \cite{L2020,S1993}), where the region for a chemical substance to diffuse across is significantly larger compared to that of a reaction process \cite{B1999,Z2012}.  We focus on  convection--diffusion equation~\eqref{equl} satisfying a {\it weak convection effect} or {\it large reaction effect}~(cf. \cite{S1999}).  In the framework of our analysis, both effects can be considered as the following mathematical settings:
 \begin{align}\label{ab619}
 |\abar(x)|^2< 4h(x),\,\,x\in\mathbb{R}^N\setminus\{0\},\,\,\text{and}\,\,\ds\lim_{\lambda\to\infty}\lambda^{\kappa}\min_{\overline{\Omega}}\left(4h(\lambda z)-|\abar(\lambda z)|^2\right)>0\,\,\text{for\,\,some}\,\,\kappa\in(0,2).
\end{align}

Equation~\eqref{equl}--\eqref{bdu} with assumption \eqref{ab619} fulfills some well-known convection--diffusion equations. In particular, \eqref{ab619} includes the case $\abar(x)\equiv\vec{0}$ and $h(x)>0$, which {arises} in numerous applications such as the homogeneous chemical reactions. \eqref{ab619} is also a case when considering a convection--diffusion  equation~\eqref{equl} with a large reaction $h(x)$. For example, $\min\limits_{\overline{\lomega}}h$ is sufficiently large so that $\max\limits_{\overline{\lomega}}|\abar|^2<4\min\limits_{\overline{\lomega}}h$. The latter assumption in \eqref{ab619} includes the case $\ds\min_{\overline{\Omega}}(4h(\lambda z)-|\abar(\lambda z)|^2)\sim\lambda^{-\widehat{\kappa}}\xrightarrow{\lambda\gg1}0$ for $0<\widehat{\kappa}<\kappa$, where the region of $\kappa$ is optimal in our analysis.

Under \eqref{ab619}, equation \eqref{equl} with the standard Dirichlet boundary conditions on $\partial\lomega$ has a unique solution for any $\lambda>0$. However, the situation becomes quite different for the non-local type boundary condition~\eqref{bdu}. The issue about the range of $\lambda$ for the uniqueness of \eqref{equl}--\eqref{bdu} seems to be a challenge. In a special case when $\lomega$ is an annulus, we provide an example to explain that there  exists $\lambda=\lambda^*>0$ such that equation~\eqref{equl}--\eqref{bdu} has infinitely numerous solutions or no solution.

\begin{example}\label{ex1}
 We consider the following equation in the domain~$\mathbb{A}_{\lambda;2\lambda}:=\left\{x\in\mathbb{R}^2:\,\lambda<|x|<2\lambda\right\}$, i.e. the domain $\lomega$ defined in \eqref{domain} with $\Omega_{\mathrm{in}}=\{x:|x|<1\}$ and $\Omega_{\mathrm{out}}=\{x:|x|<2\}$ in $\mathbb{R}^2$:
\begin{align}\label{Buu}
\begin{cases}
\ds-\Delta\ue(x)+\frac{x}{|x|^2}\cdot\nabla\ue(x)+\ue(x)=0,\quad\,x\in\mathbb{A}_{\lambda;2\lambda},\vspace{6pt}\\
    \ds\ue(x)= \bi+ \mathfrak{g}\int_{\mathbb{A}_{\lambda;2\lambda}}\ue(y)\dy,\quad\mathrm{for}\quad|x|=\lambda;\qquad\ue(x)= 0,\quad\mathrm{for}\quad|x|=2\lambda,
\end{cases}
\end{align}
where $\mathfrak{g}>0$ is a constant.  We shall apply  Fourier expansions to demonstrate that all solutions~$\ue$ of \eqref{Buu} are radially symmetric. For $x=r\mathrm{e}^{\sqrt{-1}\theta}$ with $\lambda<r<2\lambda$ and $0\leq\theta<2\pi$, we let 
\[\ue(r\mathrm{e}^{\sqrt{-1}\theta})=\frac{1}{2}u_{1,0}(r)+\sum_{k=1}^\infty \left(u_{1,k}(r)\cos(k\theta)+u_{2,k}(r)\sin(k\theta)\right),\]
where
\[u_{1,k}(r):=\frac{1}{\pi}\int_{-\pi}^{\pi} \ue (r\mathrm{e}^{\sqrt{-1}\theta})\cos(k\theta)\,\mathrm{d}\theta, \quad  u_{2,k}(r):=\frac{1}{\pi}\int_{-\pi}^{\pi} \ue (r\mathrm{e}^{\sqrt{-1}\theta})\sin(k\theta)\,\mathrm{d}\theta.\] 
Then, one may check that for $i=1,2$, and $k\in\mathbb{N}$, $u_{i,k}$  satisfies the  equation
\[ u_{i,k}''(r)-\left(1+\frac{k^2}{r^2}\right)u_{i,k}(r)=0\,\,\mathrm{in}\,\,(\lambda,2\lambda), \quad u_{i,k}(\lambda)=u_{i,k}(2\lambda)=0. \]
Such an equation only has a trivial solution $u_{i,k}(r) \equiv 0$, and we obtain that $\ue(x)=\frac{1}{2}u_{1,0}(|x|)$ is radially symmetric. As a consequence, we can solve \eqref{Buu}  and verify the non-existence, uniqueness and multiplicity of the solutions  $\ue$'s corresponding to $\bi$ and $\lambda$. More accurately, we completely classify the solutions of \eqref{Buu} as follows:
\begin{itemize}
 \item[\bf (i).] When $\bi=0$ and $\lambda^*:=\lambda^*(\mathfrak{g})>0$ satisfies\footnote{Let $f(\lambda)= \left(\frac{1+\lambda}{2}-\frac{1}{4\mathfrak{g}\pi}\right)\frac{\mathrm{e}^{2\lambda}-1}{\lambda}-2(\mathrm{e}^{\lambda}-1)$. Since $\mathfrak{g}>0$, we have $\lim_{\lambda\to0+}f(\lambda)=1-\frac{1}{2\mathfrak{g}\pi}<1$ and $\lim_{\lambda\to\infty}f(\lambda)=\infty$. Along with the intermediate value theorem of continuous functions, we verify that \eqref{alg-eq} has a positive root $\lambda^*=\lambda^*(\mathfrak{g})$.}
\begin{align}\label{alg-eq}
    \left(\frac{1+\lambda^*}{2}-\frac{1}{4\mathfrak{g}\pi}\right)\frac{\mathrm{e}^{2\lambda^*}-1}{\lambda^*}-2(\mathrm{e}^{\lambda^*}-1)=1,
\end{align}
 for any constant $c^*$, all $u_{\lambda^*,c^*}(x)=c^*\left(\mathrm{e}^{|x|}-\mathrm{e}^{4\lambda^*-|x|}\right)$ are solutions of \eqref{Buu} with $\lambda=\lambda^*$. 
\item[\bf (ii).] When $\bi \ne 0$ and $\lambda=\lambda^*$, \eqref{Buu} has no solution.
\item[\bf (iii).] If $\lambda>0$ and $\lambda\neq\lambda^*$, then for each $\bi\in\mathbb{R}$, \eqref{Buu} has a unique solution. 
\end{itemize}
\end{example}
Simple calculations can be employed to verify this result. We thus omit the details here.  For general cases, the issues of the non-existence, uniqueness, and multiplicity of solutions of \eqref{equl}--\eqref{bdu} become more complicated; see Proposition~\ref{m-cor}. What we shall point out is that the uniqueness of solutions to \eqref{equl}--\eqref{bdu} in $\lomega$ defined by \eqref{domain} exactly depends variously on $\lambda$, $\bi$, and $\bo$.

Additionally, to obtain a more detailed uniqueness result of \eqref{equl}--\eqref{bdu} with large $\lambda$, for $\gfi$ and $\gfo$, we made the following precise assumption:
\begin{align}\label{gio}
\begin{cases}
 \ds\lim_{\lambda\to\infty}\lambda^{N+\frac{\kappa-2}{2}}\max_{z\in\overline{\Omega}}|g(\lambda z)|=0,\\
 \ds\lim_{|z|\to0}|z|^Ng(z)=0,
\end{cases}
 \,\,\text{for}\,\,g=\gfi,\,\gfo,
\end{align}
where~$\kappa\in(0,2)$ has been defined in \eqref{ab619}. Assumption~\eqref{gio} allows  $g$ to blow up at $z=0$. An example for \eqref{gio} is $g(z)=|z|^{-N+\frac{2-{\kappa}^*}{2}}$ with ${\kappa}^*\in(\kappa,2)$.  Although \eqref{gio} implies $\left|\int_{\lomega}g(y)\dy\right|=\lambda^N\left|\int_{\Omega}g(\lambda z)\,\dz\right|\ll\lambda^{\frac{2-\kappa}{2}}$  as $\lambda\to\infty$,  we emphasise that it also includes the case $\lim\limits_{\lambda\to\infty}\left|\int_{\lomega}g(y)\dy\right|=\infty$. 

\subsection{The main result}
 We introduce the equations corresponding  to \eqref{equl} with the standard Dirichlet boundary conditions before stating the main results. Let $\phl$ be a solution of the equation
\begin{align}\label{eqphl}
\begin{cases}
-\Delta\phl(x)+\abar(x)\cdot\nabla\phl(x)+h(x)\phl(x)=0\quad\mathrm{in}\quad \lomega,\\
\phl=1\,\,\mathrm{on}\,\,\gi,\qquad\phl=0\,\,\mathrm{on}\,\,\go,
\end{cases}
\end{align}
and let $\psl$ be a solution of the equation
\begin{align}\label{eqpsl}
\begin{cases}
-\Delta\psl(x)+\abar(x)\cdot\nabla\psl(x)+h(x)\psl(x)=0\quad\mathrm{in}\quad \lomega,\\
\psl=0\,\,\mathrm{on}\,\,\gi,\qquad\psl=1\,\,\mathrm{on}\,\,\go.
\end{cases}
\end{align}
Note that $\lomega$ is a bounded smooth domain for each fixed $\lambda>0$, and the constraints  $\abar\big|_{\overline{\lomega}}$ and $h\big|_{\overline{\lomega}}$ are H\"{o}lder continuous with exponent $\alpha\in(0,1)$. As a consequence, by the standard elliptic regularity theory (cf. \cite[Theorem~6.14]{GT2001}), we obtain that the equations \eqref{eqphl} and \eqref{eqpsl} have unique solutions~$\phl,\,\psl\in\text{C}^{2,\alpha}(\overline{\lomega})$, where the uniqueness is trivially due to  $h\big|_{\overline{\lomega}}>0$.

Let $c_1$ and $c_2$ be constants. Then, by the uniqueness of {solutions of}  \eqref{eqphl} and \eqref{eqpsl}, the solution of equation~\eqref{equl} with boundary conditions $\ue=c_1$ on $\gi$ and $\ue=c_2$ on $\go$ is uniquely expressed by $\ue=c_1\phl+c_2\psl$.  Accordingly,  for each solution $\ue$ of \eqref{equl}--\eqref{bdu}  (if it exists), we have an implicit representation for $\ue$ given by
\begin{align}\label{imp-u}
      \ue(x)=\left(\bi+  \int_{\lomega}\gfi(y)\ue(y)\dy\right)\phl(x)+\left(\bo+  \int_{\lomega}\gfo(y)\ue(y)\dy\right)\psl(x),\quad\,x\in\overline{\lomega}.
  \end{align}
Here we perform an argument that is different from the fixed point approach. We will show that {the} two non-local coefficients (depending on unknown $\ue$) in the right-hand side of \eqref{imp-u} can be explicitly  expressed by $\phl$ and $\psl$. In Proposition~\ref{m-cor}, we establish a criterion of the existence and uniqueness for \eqref{equl}--\eqref{bdu}.
\begin{proposition}\label{m-cor}
Let $\lomega$ be defined by \eqref{domain}. Assume that $\bi$ and $\bo$ are constants independent of $\lambda$, {$\gfi,\,\gfo\in\mathrm{C}(\mathbb{R}^N\setminus\{0\};\mathbb{R})$ satisfy \eqref{gio}, and $\abar\in\mathrm{C}^{\alpha}(\mathbb{R}^N\setminus\{0\};\mathbb{R}^N)$ and $h\in\mathrm{C}^{\alpha}(\mathbb{R}^N\setminus\{0\};\mathbb{R})$ satisfy \eqref{ab619}, where $\alpha\in(0,1)$.} Then we have the following:
\begin{itemize}
    \item[(i)] If $\lambda$ satisfies $\det(\mathcal{I}- \mathcal{R}_{\lambda})\neq0$, then \eqref{equl}--\eqref{bdu} has a unique solution  satisfying
  \begin{align}\label{sol-u}
\fbox{\begin{minipage}{316pt}
$\,\,\,\ue(x)={\det\left((\mathcal{I}-\mathcal{R}_{\lambda})^{-1}\mathcal{C}_{\psl}\right)}\phl(x)+{\det\left((\mathcal{I}-\mathcal{R}_{\lambda})^{-1}\mathcal{C}_{\phl}\right)}\psl(x),$
\end{minipage}}
  \end{align}
 where
\begin{align}\label{big-i-a}
\mathcal{I}=\begin{bmatrix}
1 & 0 \vspace{3pt}\\
0 & 1
\end{bmatrix},\,\,
   \mathcal{R}_{\lambda}=   \begin{bmatrix}
\ds\int_{\lomega}\gfi(y)\phl(y)\dy& \ds\int_{\lomega}\gfi(y)\psl(y)\dy \vspace{3pt}\\
\ds\int_{\lomega}\gfo(y)\phl(y)\dy & \ds\int_{\lomega}\gfo(y)\psl(y)\dy 
\end{bmatrix}
\end{align}
and
\begin{align}\label{big-ii}
\mathcal{C}_{\psl}= \begin{bmatrix}
\bi& \ds\quad-  \int_{\lomega}\gfi(y)\psl(y)\dy \vspace{3pt}\\
\bo & \ds1-  \int_{\lomega}\gfo(y)\psl(y)\dy 
\end{bmatrix},\,\,
   \mathcal{C}_{\phl}= \begin{bmatrix}
\ds1-  \int_{\lomega}\gfi(y)\phl(y)\dy& \bi \vspace{3pt}\\
\ds\quad-  \int_{\lomega}\gfo(y)\phl(y)\dy & \bo 
\end{bmatrix}.
\end{align} 
    \item[(ii)] If $\lambda$, $\bi$, and $\bo$ satisfy $\det(\mathcal{I}- \mathcal{R}_{\lambda})=0$ and $\det\mathcal{C}_{\psl}=0$, then \eqref{equl}--\eqref{bdu} has infinitely many solutions.
    \item[(iii)] If $\lambda$, $\bi$, and $\bo$ satisfy  $\det(\mathcal{I}- \mathcal{R}_{\lambda})=0$ and $\det\mathcal{C}_{\psl}\neq0$, then \eqref{equl}--\eqref{bdu} has no solution.
\end{itemize}
\end{proposition}

Having Proposition~\ref{m-cor} at hand, we are in a position to state the main result.

\begin{theorem}\label{m-thm}
Under the same assumptions as in Proposition~\ref{m-cor},  there exist $0<\lambda_*<\lambda^*$ such that for each $\lambda\in(0,\lambda_*)\cup(\lambda^*,\infty)$, we have $\det(\mathcal{I}- \mathcal{R}_{\lambda})>0$, and \eqref{equl}--\eqref{bdu} has a unique solution~$\ue$ satisfying~\eqref{sol-u}. Moreover, as $\lambda\gg1$, there exist positive constants $C^*$ and $M^*$ independent of $\lambda$ such that
  \begin{align}\label{u-in-e}
|\ue(x)|\leq\,C^*\exp\left(-M^*\lambda^{-\frac{\kappa}{2}}\mathrm{dist}(x,{\partial\lomega})\right),\quad\mathrm{for\,\,all}\,\,x\in   \overline{\lomega}.   
\end{align}
\end{theorem}
Figure~\ref{fig1} is a fundamental understanding of Theorem~\ref{m-thm}. In Section~\ref{sec-pf}, we will prove Proposition~\ref{m-cor} and Theorem~\ref{m-thm}. 

\begin{remark}
\eqref{u-in-e} shows that for an interior point $x\in\Omega_{\lambda}$ satisfying $\lim\limits_{\lambda\to\infty}\mathrm{dist}(\frac{x}{\lambda},{\partial\Omega})>0$,  $\lim\limits_{\lambda\to\infty}|\ue(x)|=0$ holds because for $\kappa\in(0,2)$,
\begin{align*}
    |\ue(x)|\leq\,C^*\exp\left(-M^*\lambda^{-\frac{\kappa}{2}}\mathrm{dist}(x,{\partial\lomega})\right)=C^*\exp\left(-M^*\lambda^{1-\frac{\kappa}{2}}\mathrm{dist}(\frac{x}{\lambda},{\partial\Omega})\right)\ll1.
\end{align*}
\end{remark}

\begin{remark}[The ${H}^1$-capacity of $\lomega$]
Inspired by \eqref{sol-u}, which presents a linear combination of $\phl$ and $\psl$ for solution~$\ue$, the ${H}^1$-capacity of $\lomega$  is naturally considered since it is a relevant measure of the ``size'' for $\lomega$ supporting Dirichlet boundary conditions   (cf. \cite{BM2006,M1985}). (Whereas notice  the boundary conditions of $\phl$ and $\psl$ in \eqref{eqphl} and \eqref{eqpsl}.) Recall the ${H}^1$-capacity of $\lomega$:
\begin{equation}\label{cap-e}
 \mathrm{cap}(\lomega)=\min\left\{\int_{\lomega}|\nabla U|^2\dx:\,U\in {H}^1(\lomega),\,\,U=1\,\,\mathrm{on}\,\,\gi,\,\,U=0\,\,\mathrm{on}\,\,\go\right\},  
\end{equation}
where the minimum is achieved when $U$ is harmonic in $\lomega$. In particular, $\lomega$ has the~property
\begin{equation*}
   \mathrm{cap}(\lomega)=\lambda^{N-2}\mathrm{cap}(\Omega). 
\end{equation*}
 When $N=2$, $\mathrm{cap}(\lomega)$ is independent of $\lambda$. As a consequence, the effect of $\mathrm{cap}(\lomega)$ on the existence and uniqueness of the solutions of the equation~\eqref{equl}--\eqref{bdu} appears insignificant. However, when $N\geq3$, Theorem~\ref{m-thm} can be reinterpreted that there exist $0<L_*<L^*$ depending primarily on $\mathrm{cap}(\Omega)$ and dimension $N$ such that if $\mathrm{cap}(\lomega)\in(0,L_*)\cup(L^*,\infty)$, the equation~\eqref{equl}--\eqref{bdu} has a unique solution. For $\lomega$ replaced with the general domains, the {\bf domain capacity effect} on the uniqueness issue of \eqref{equl}--\eqref{bdu} seems more subtle. The main difficulty, in our perspective, lies in obtaining the refined estimates (with respect to $\lambda$) of the non-local terms in \eqref{bdu}. The general theory for this problem is highly nontrivial, which is a direction of our research in the  future.
\end{remark}

Theorem~\ref{m-thm} focuses  primarily on the uniqueness result of  equation~\eqref{equl}--\eqref{bdu} in the annular-like domains $\lomega$ defined by \eqref{domain}. We shall stress that the argument can straightforwardly be generalized to the same equation as \eqref{equl} in a bounded domain with several boundary components (see Theorem~\ref{thm2} in Section~\ref{sec-mb}).

\subsection{An example contrasting to Theorem~\ref{m-thm}}
In this section, we will emphasise that the property of the domain~$\lomega$ plays a critical role in the uniqueness result presented in Theorem~\ref{m-thm}.

 Proposition~\ref{m-cor} presents that the equation~\eqref{equl}--\eqref{bdu} has a unique solution if and only if $\mathcal{I}-\mathcal{R}_{\lambda}$ is invertible, where  all elements in $\mathcal{R}_{\lambda}$ are associated with the integral-type boundary condition~\eqref{bdu}, which involves the domain $\lomega$. Note that the result of  Proposition~\ref{m-cor} still holds for general expanding domains because the property of the domain~$\lomega$ is not required in the proof (cf. Section~\ref{sec-3_1}). We emphasise that Theorem~\ref{m-thm} holds for the domain $\lomega$ defined by \eqref{domain}, but it may not hold for the general expanding domains. We would provide an example to show that Theorem~\ref{m-thm} may not hold when the expanding domain has a fixed inner boundary that is independent of $\lambda$.

 For a more detailed explanation, let us recall Example~\ref{ex1}, where all domains~$\mathbb{A}_{\lambda;2\lambda}=\{x\in\mathbb{R}^2:\,\lambda<|x|<2\lambda\}$ have the ${H}^1$-capacity $\text{cap}\left(\mathbb{A}_{\lambda;2\lambda}\right)=\frac{2\pi}{\ln2}$, which is independent of $\lambda>0$ (see \eqref{cap-e} and the footnote\footnote{We refer the readers to \cite{BM2006} (see also, \eqref{cap-e} with $N=2$) for the ${H}^1$-capacity of annular domains in $\mathbb{R}^2$. Let $\mathbb{A}_{R_1;R_2}:=\{x\in\mathbb{R}^2:\,R_1<|x|<R_2\}$. Then, the ${H}^1$-capacity of $\mathbb{A}_{R_1;R_2}$ is $\text{cap}\left(\mathbb{A}_{R_1;R_2}\right)=\frac{2\pi}{\ln(R_2/R_1)}$ which can be obtained when $U(x)=\frac{\ln{R_2}-\ln|x|}{\ln{R_2}-\ln{R_1}}$, $x\in\overline{\mathbb{A}_{R_1;R_2}}$.}). In this case,  $\lambda=\lambda^*$ (see \eqref{alg-eq}) for the non-existence/non-uniqueness of the solutions of equation~\eqref{Buu} depends mainly on the coefficient $\mathfrak{g}$ of the integral-type boundary condition. In the following Example~\ref{ex2}, we adopt the same equation as  in Example~\ref{ex1}, but the domain $\mathbb{A}_{\lambda;2\lambda}$ is replaced by the annular domain $\mathbb{A}_{1;\lambda}:=\{x\in\mathbb{R}^2:\,1<|x|<\lambda\}$ with an inner boundary independent of $\lambda$. A difference between $\mathbb{A}_{\lambda;2\lambda}$ and $\mathbb{A}_{1;\lambda}$ with $\lambda$ varying comes from the fact that $\text{cap}\left(\mathbb{A}_{1;\lambda}\right)=\frac{2\pi}{\ln\lambda}$ depends on $\lambda>1$. Because the inner boundary of $\mathbb{A}_{1;\lambda}$ is fixed, assumption~\eqref{gio} is not effective for $g=\gfi$ on $|x|=1$. This observation prompts us to consider the boundary condition $\ue=\bi+\int_{\mathbb{A}_{1;\lambda}}\gfi(y)\ue(y)\,\dy$ with a precise $\gfi$ (see \eqref{moon-c0}) on the inner boundary $|x|=1$. We will show in Case~2(II) of Example~\ref{ex2}  that when $\bi=0$ ($\bi\neq0$, respectively), there exists a strictly increasing sequence $\{\lambda_k\}_{k\in\mathbb{N}}$ with $\lambda_k\xrightarrow{k\to\infty}\infty$ such that the equation in the domain $\mathbb{A}_{1;\lambda_k}$ has infinitely many solutions (no solution, respectively). 
\begin{example}\label{ex2}
Let $N=2$. For $\lambda > 1$, we consider the equation
\begin{align}\label{Buu1}
\begin{cases}
\ds-\Delta\ue(x)+\frac{x}{|x|^2}\cdot\nabla\ue(x)+\ue(x)=0,\quad \,x\in\mathbb{A}_{1;\lambda},\vspace{6pt}\\
    \ds\ue(x)= \bi+ \int_{\mathbb{A}_{1;\lambda}}\gfi(y)\ue(y)\dy,\quad\mathrm{for}\quad|x|=1;\qquad\ue(x)= 0,\quad\mathrm{for}\quad|x|=\lambda,
\end{cases}
\end{align}
where
\begin{equation}\label{moon-c0}
    \gfi(y)=C_0\frac{\text{e}^{-|y|}}{|y|}\sin(|y|)
\end{equation}
obeys the assumption \eqref{gio}, and $C_0$ is a constant.

We will apply Proposition~\ref{m-cor} to equation~\eqref{Buu1} since its consequence still holds for equation~\eqref{equl}--\eqref{bdu} when the domain $\lomega$ is replaced with $\mathbb{A}_{1;\lambda}$. First, one solves the corresponding equations \eqref{eqphl} and \eqref{eqpsl} with $\abar(x)=\frac{x}{|x|^2}$ and $h(x)=1$:
\begin{equation}\label{new-eqp}
 \phi_\lambda(x)=\frac{\text{e}^{-|x|}-\text{e}^{-2\lambda+|x|}}{\text{e}^{-1}-\text{e}^{-2\lambda+1}}, \ \ \psi_\lambda(x)=\frac{\text{e}^{|x|}-\text{e}^{2-|x|}}{\text{e}^\lambda-\text{e}^{2-\lambda}}. 
\end{equation}
Next, we shall determine a constant $C_0$ in \eqref{moon-c0} so that Proposition~\ref{m-cor}(ii) or (iii) occurs for ``infinitely many'' $\lambda$. By \eqref{new-eqp}, a simple calculation yields
\begin{equation}\label{moon-g}
   \begin{aligned} 
\int_{\mathbb{A}_{1;\lambda}}&\, \gfi(y)\phi_\lambda(y)\,\dy
=  \frac{2\pi C_0}{\text{e}^{-1}-\text{e}^{-2\lambda+1}} \int_1^\lambda (\text{e}^{-r}-\text{e}^{-2\lambda+r}) \text{e}^{-r}\sin r\, \text{d}r \\
=& \frac{2\pi C_0}{\text{e}^{-1}-\text{e}^{-2\lambda+1}} \left[\frac{\text{e}^{-2}}{5}(2\sin 1+\cos 1)- \text{e}^{-2\lambda}\left(\cos 1 -\cos \lambda +\frac15 (2\sin\lambda+\cos\lambda)\right)\right],
\end{aligned} 
\end{equation}
which implies 
\begin{equation}\label{moon-l}
    \lim_{\lambda\to\infty}\int_{\mathbb{A}_{1;\lambda}} \gfi(y)\phi_\lambda(y)\,\dy=\frac{2\pi C_0 }{5 \text{e}}(2\sin 1+\cos 1).
\end{equation}
On the other hand, by \eqref{big-i-a} with $\lomega=\mathbb{A}_{1;\lambda}$ and $\gfo=0$, we have
\begin{equation}\label{moon-d}
    \det(\mathcal{I}- \mathcal{R}_{\lambda})=1-\int_{\mathbb{A}_{1;\lambda}} \gfi(y)\phi_\lambda(y)\,\dy.
\end{equation}
As a consequence, we have the following results:\\
{\bf Case~1.} If $C_0\neq\ds \frac{5\text{e}}{2\pi(2\sin 1+\cos 1)}$, then by \eqref{moon-l}--\eqref{moon-d}, we have $\det(\mathcal{I}- \mathcal{R}_{\lambda})\neq0$ as $\lambda\gg1$. Applying Proposition~\ref{m-cor}(i) to this case, we verify that as $\lambda\gg1$, \eqref{Buu1} has a unique solution. \vspace{3pt}\\ 
{\bf Case~2.} If $C_0 =\ds \frac{5\text{e}}{2\pi(2\sin 1+\cos 1)}$, then by \eqref{moon-g} and \eqref{moon-d}, we obtain that (see the Appendix):
\begin{equation}\label{apid}
    \left\{ \lambda>1:\,\det(\mathcal{I}- \mathcal{R}_{\lambda})=0\right\}=\left\{2k\pi +1,\,(2 k-1)\pi +2\theta_0-1:\,k\in\mathbb{N}\right\}:=\boldsymbol{\mathsf{S}}_0,
\end{equation}
where $\theta_0=\arcsin\frac{2}{\sqrt{5}}\in(1,\frac{\pi}{2})$. Thus, by Proposition~\ref{m-cor}(i)--(iii) and \eqref{moon-d}--\eqref{apid}, we have the following results:
\begin{itemize}
    \item[\bf (I)] If $\lambda\not\in\boldsymbol{\mathsf{S}}_0$, then  $\mathcal{I}-\mathcal{R}_{\lambda}$ is  invertible. Therefore,  equation~\eqref{Buu1} has a unique solution. 
      \item[\bf (II)] If $\lambda\in\boldsymbol{\mathsf{S}}_0$, then for $\bi\neq0$, \eqref{Buu1} has no solution; for $\bi=0$,  \eqref{Buu1} has infinitely many solutions~$\ue=c_*\phl$, where $\phl$ is defined by \eqref{new-eqp} and $c_*\in\mathbb{R}$ is arbitrary.
\end{itemize} 
\end{example}

The rest of this paper is structured as follows. In Section~\ref{pre2}, we establish estimates of $\phl$ and $\psl$ in $\overline{\lomega}$ for $\lambda\gg1$, which is crucial for the proof of  Theorem~\ref{m-thm}. In Section~\ref{sec-pf}, we state the proof of Proposition~\ref{m-cor} and Theorem~\ref{m-thm} and provide a remark for the uniqueness under an approach of the maximum principle. Next, as a generalization of Theorem~\ref{m-thm}, in Section~\ref{sec-mb}, we consider  equation~\eqref{equl} in an expanding domain with numerous boundary components. Under the corresponding integral-type boundary conditions, we establish the uniqueness result which will be introduced in Theorem~\ref{thm2}.  We make the concluding remarks and future problems in Section~\ref{sec-open}. Finally, in the Appendix we state the proof of \eqref{apid} for the sake of clarity and completeness.

\section{\bf Preliminary estimates of $\phl$ and $\psl$}\label{pre2}
In this section, we first establish the required estimates of $\phl$ and $\psl$ for $\lambda\gg1$, where $\phl$ and $\psl$ are solutions of equations~\eqref{eqphl} and \eqref{eqpsl}, respectively.  Let $(\Phi_{\lambda}(z),\Psi_{\lambda}(z))=(\phl(x),\psl(x))$ with $z=\lambda^{-1}x\in\Omega$. Then $\Phi_{\lambda}(z)$ and $\Psi_{\lambda}(z)$ satisfy
\begin{align}\label{big-phi}
\begin{cases}
\ds-\frac{1}{\lambda^2}\Delta_z\Phi_{\lambda}(z)+ \frac{1}{\lambda}\abar(\lambda z)\cdot\nabla_z\Phi_{\lambda}(z)+h(\lambda z)\Phi_{\lambda}(z)=0\quad\mathrm{in}\quad \Omega,\\
    \Phi_{\lambda}=1\quad\mathrm{on}\,\,\partial\Omega_{\mathrm{in}},\qquad\Phi_{\lambda}=0\quad\mathrm{on}\,\,\partial\Omega_{\mathrm{out}},
\end{cases}
\end{align}
and
\begin{align}\label{big-psi}
\begin{cases}
\ds-\frac{1}{\lambda^2}\Delta_z\Psi_{\lambda}(z)+ \frac{1}{\lambda}\abar(\lambda z)\cdot\nabla_z\Psi_{\lambda}(z)+h(\lambda z)\Psi_{\lambda}(z)=0\quad\mathrm{in}\quad \Omega,\\
    \Psi_{\lambda}=0\quad\mathrm{on}\,\,\partial\Omega_{\mathrm{in}},\qquad\Psi_{\lambda}=1\quad\mathrm{on}\,\,\partial\Omega_{\mathrm{out}},
\end{cases}
\end{align}
respectively. Since $h>0$ (by \eqref{ab619}), the maximum principle implies $0\leq\Phi_{\lambda},\,\Psi_{\lambda}\leq1$ in $\Omega$.

We first establish an interior estimate for $\Phi_{\lambda}$. Multiplying both sides of the equation in \eqref{big-phi} by $\Phi_{\lambda}$, one may check from \eqref{ab619} that, if $\lambda>0$ is sufficiently large, there exists a positive constant $M$ independent of $\lambda$ such that
\begin{align*}
   \frac{1}{\lambda^2}\Delta_z\Phi_{\lambda}^2=&\, \frac{2}{\lambda^2}|\nabla_z\Phi_{\lambda}|^2+\frac{2}{\lambda}\abar(\lambda z)\Phi_{\lambda}\cdot\nabla_z\Phi_{\lambda}+2h(\lambda z)\Phi_{\lambda}^2\\
   \geq&\,2\left(h(\lambda z)-\frac{|\abar(\lambda z)|^2}{4}\right)\Phi_{\lambda}^2\geq\lambda^{-\kappa}M^2\Phi_{\lambda}^2,
\end{align*}
i.e.,
\begin{align}\label{es-l-ph}
  \Delta_z\Phi_{\lambda}^2\geq\lambda^{2-\kappa}M^2\Phi_{\lambda}^2\qquad\mathrm{in}\quad\Omega,  
\end{align}
where we notice $\lambda^{2-\kappa}\gg1$ since $0<\kappa<2$.

Let $\mu_1>0$ be the principal eigenvalue of $-\Delta_z$ in ${H}^1_0(\Omega)$, and let $\varphi_1$ be the corresponding positive eigenfunction with $||\varphi_1||_{\mathrm{L}^\infty(\Omega)}=1$. Consider the following auxiliary function:
\begin{align}\label{def-Upsilon}
    \mathcal{V}_{\lambda}(z) := \exp\left(-{M_1}\lambda^{1-\frac{\kappa}{2}}\varphi_1(z) \right)\qquad\mathrm{in}\quad\overline{\Omega}.
\end{align}
We will determine a positive constant $M_1$ such that \eqref{def-Upsilon} is a superposition of \eqref{es-l-ph}. First, it is easy to check that $\mathcal{V}_{\lambda}$ satisfies
\begin{align*}
  \Delta_z\mathcal{V}_{\lambda} =\lambda^{2-\kappa}\left(M_1\lambda^{\frac{\kappa}{2}-1}\mu_1 \varphi_1  +M_1^2 |\nabla \varphi_1|^2\right) \mathcal{V}_{\lambda}.  
\end{align*}
Since $0<\lambda^{\frac{\kappa}{2}-1}\ll1$, and $|\nabla \varphi_1|$ is bounded in $\overline{\Omega}$ (cf. \cite{GT2001}), one can choose a suitable $M_1>0$ such that $\Delta_z \mathcal{V}_{\lambda} \leq \lambda^{2-\kappa}M^2 \mathcal{V}_{\lambda}$ in $\Omega$ for sufficiently large $\lambda$. As a consequence, by combining this differential inequality with \eqref{es-l-ph}, we arrive at
\begin{align}\label{diff-ineq}
   \Delta_z(\Phi_{\lambda}^2- \mathcal{V}_{\lambda}) \geq \lambda^{2-\kappa}M^2(\Phi_{\lambda}^2- \mathcal{V}_{\lambda})\qquad\mathrm{in}\quad \Omega.
\end{align}
Note also that $\Phi_{\lambda}^2\leq1= \mathcal{V}_{\lambda}$ on $\partial\Omega$, i.e.
\begin{align}\label{bd-ineq}
 \Phi_{\lambda}^2- \mathcal{V}_{\lambda}\leq0\qquad\mathrm{on}\quad\partial\Omega.   
\end{align}
  Applying the standard comparison theorem to \eqref{diff-ineq}--\eqref{bd-ineq}, we obtain
 \begin{align}\label{v-Up}
 \Phi_{\lambda}^2(z)\leq \mathcal{V}_{\lambda}(z),\quad\mathrm{for\,\,all}\,\,z\in   \overline{\Omega}. 
 \end{align}

Now we claim that there exists a {positive} constant $M^*$ independent of $\lambda$ such that
\begin{align}\label{Upsilon-y}
    \mathcal{V}_{\lambda}(z)\leq\exp\left(-{M^*}\lambda^{1-\frac{\kappa}{2}}\mathrm{dist}(z,{\partial\Omega})\right),\,\,\mathrm{for}\,\,z\in\Omega.
\end{align}
Since $\partial\Omega$ is  smooth, the eigenfunction $\varphi_1$ is both bounded above and below by strictly positive multiples of $\mathrm{dist}(x,\partial\Omega)$ (see, e.g. \cite[Section 3]{D1987}). Here, for the sake of completeness, we offer an alternative proof for a required lower bound of $\varphi_1$ as follows. Since $\varphi_1$ is smooth and strictly positive in $\Omega$ and  $\varphi_1=0$ on $\partial\Omega$, by the Hopf's lemma, the outward normal derivative of $\varphi_1$ on $\partial\Omega$ is strictly negative, implying that $\varphi_1(x) \geq {M_2} \mathrm{dist}(x,\partial\Omega)$ in $\overline{\Omega_{d_2}}:=\{x\in\Omega:\mathrm{dist}(x,\partial\Omega)\leq d_2\}$, for a small $d_2>0$ and a positive constant $M_2$ depending on $d_2$, $\|\varphi_1\|_{\text{C}^2(\overline{\Omega_{d_2}})}$ and $\ds\min_{\partial\Omega}|\partial_{\vec{n}}\varphi_1|$. As a consequence,
\begin{align*}
 \inf_{z\in\Omega}\frac{\varphi_1(z)}{\mathrm{dist}(z,\partial\Omega)} \geq\min_{\overline{\Omega\setminus\Omega_{d_2}}}\left\{M_2,\frac{\varphi_1}{\mathrm{diam}(\Omega)}\right\}. 
\end{align*}
Along with \eqref{def-Upsilon}, we arrive at \eqref{Upsilon-y} with $M^*=\frac12M_1\min_{\overline{\Omega\setminus\Omega_{d_2}}}\left\{M_2,\frac{\varphi_1}{\mathrm{diam}(\Omega)}\right\}$. By \eqref{v-Up} and \eqref{Upsilon-y}, we have $0\leq\Phi_{\lambda}(z)\leq \exp\left(-M^*\lambda^{1-\frac{\kappa}{2}}\mathrm{dist}(z,{\partial\Omega})\right),\,\,\forall\,z\in   \overline{\Omega}$. That is,
 \begin{align}\label{ph-*}
 0\leq\phl(x)\leq \exp\left(-M^*\lambda^{-\frac{\kappa}{2}}\mathrm{dist}(x,{\partial\lomega})\right),\quad\mathrm{for\,\,all}\,\,x\in   \overline{\lomega}. 
 \end{align}
 Similarly, by \eqref{ab619},  \eqref{big-psi} and the fact that $0\leq\psl\leq1$ on the boundary, we also have
  \begin{align}\label{ps-*}
 0\leq\psl(x)\leq \exp\left(-M^*\lambda^{-\frac{\kappa}{2}}\mathrm{dist}(x,{\partial\lomega})\right),\quad\mathrm{for\,\,all}\,\,x\in   \overline{\lomega}. 
 \end{align}
 Estimates \eqref{ph-*} and \eqref{ps-*} play a crucial role in the proof of Theorem~\ref{m-thm}.
 \begin{remark}
 An alternative proof of  \eqref{ph-*} and \eqref{ps-*} can be found in \cite[Proposition 2]{L2016}; see also \cite[(2.30)]{L2016}.
 \end{remark}
\section{\bf Proofs of the main results}\label{sec-pf}
\subsection{Proof of Proposition~\ref{m-cor}}\label{sec-3_1}
For the sake of simplicity, we let $B_{\lambda,\mathrm{i}}=\bi+  \int_{\lomega}\gfi(y)\ue(y)\dy$ and $B_{\lambda,\mathrm{o}}=\bo+  \int_{\lomega}\gfo(y)\ue(y)\dy$. By \eqref{bdu} and \eqref{imp-u}, we have
\begin{align}\label{Bio}
\begin{cases}
B_{\lambda,\mathrm{i}}=\ds\bi+  \int_{\lomega}\gfi(y)\left(B_{\lambda,\mathrm{i}}\phl(y)+B_{\lambda,\mathrm{o}}\psl(y)\right)\!\text{d}y,\vspace{3pt}\\ B_{\lambda,\mathrm{o}}=\ds\bo+  \int_{\lomega}\gfo(y)\left(B_{\lambda,\mathrm{i}}\phl(y)+B_{\lambda,\mathrm{o}}\psl(y)\right)\!\text{d}y,
\end{cases}
\end{align}
which is equivalent to 
\begin{align}\label{sys-ir}
(\mathcal{I}- \mathcal{R}_{\lambda})
\begin{bmatrix}
B_{\lambda,\mathrm{i}} \vspace{3pt}\\
B_{\lambda,\mathrm{o}}
\end{bmatrix}
  =   \begin{bmatrix}
\bi \vspace{3pt}\\
\bo
\end{bmatrix},
\end{align}
where $\mathcal{I}$ and $\mathcal{R}_{\lambda}$ have been defined by \eqref{big-i-a}. 

 Accordingly, the existence, uniqueness, and multiplicity of the solutions to \eqref{equl}--\eqref{bdu} are determined by that of $(B_{\lambda,\mathrm{i}},B_{\lambda,\mathrm{o}})$ to \eqref{sys-ir}. If $\det\left(\mathcal{I}-\mathcal{R}_{\lambda}\right)\neq0$, then by applying Cramer's rule to \eqref{sys-ir}, we get a unique solution
\begin{align}\label{B-io}
    B_{\lambda,\mathrm{i}}=\det\left((\mathcal{I}-\mathcal{R}_{\lambda})^{-1}\mathcal{C}_{\psl}\right),\qquad\,B_{\lambda,\mathrm{o}}=\det\left((\mathcal{I}-\mathcal{R}_{\lambda})^{-1}\mathcal{C}_{\phl}\right),
\end{align}
where $\mathcal{C}_{\psl}$ and $\mathcal{C}_{\phl}$ have been defined by \eqref{big-ii}.  {Thus, $\det\left(\mathcal{I}-\mathcal{R}_{\lambda}\right)\neq0$ implies that \eqref{equl}--\eqref{bdu} has a unique solution $\ue$ satisfying \eqref{sol-u}.}  This completes the proof of (i). Similarly, we can prove (ii) and (iii) and complete the proof of Proposition~\ref{m-cor}.

The following corollary is a direct application of Proposition~\ref{m-cor}.
\begin{corollary}\label{cor-ch}
Under the same assumptions as in Proposition~\ref{m-cor}, we assume that $\lambda>0$ satisfies
\begin{equation}\label{c-g<1}
   \int_{\lomega}\left(|\gfi(y)|+|\gfo(y)|\right)\mathrm{d}y\leq1. 
\end{equation}
Then, $\det(\mathcal{I}- \mathcal{R}_{\lambda})>0$. In particular, equation~\eqref{equl}--\eqref{bdu} has a unique solution~$\ue$ satisfying~\eqref{sol-u}.
\end{corollary}
\begin{proof}
By \eqref{big-i-a}, we have
\begin{equation}\label{dd}
 \begin{aligned}
    \det(\mathcal{I}- \mathcal{R}_{\lambda})=&\,\left(1-\int_{\lomega}\gfi(y)\phl(y)\dy\right)\left(1-\int_{\lomega}\gfo(y)\psl(y)\dy \right)\\
    &\qquad-\int_{\lomega}\gfi(y)\psl(y)\dy\int_{\lomega}\gfo(y)\phl(y)\dy.
\end{aligned}   
\end{equation}
Let $\int_{\lomega}|\gfi(y)|\dy=A_{\lambda,\text{i}}$ and $\int_{\lomega}|\gfo(y)|\dy=A_{\lambda,\text{o}}$. Then, by \eqref{c-g<1}, we have $A_{\lambda,\text{i}}+A_{\lambda,\text{o}}\leq1$. Because $0\leq\phl,\,\psl\leq1$,  $\phl\not\equiv1$ and $\psl\not\equiv1$ in $\overline{\lomega}$, we have $\int_{\lomega}|\gfi(y)|\phl(y)\dy<A_{\lambda,\text{i}},\,\int_{\lomega}|\gfi(y)|\psl(y)\dy<A_{\lambda,\text{i}}$, $\int_{\lomega}|\gfo(y)|\phl(y)\dy<A_{\lambda,\text{o}}$ and $\int_{\lomega}|\gfo(y)|\psl(y)\dy<A_{\lambda,\text{o}}$. Thus, $1-\int_{\lomega}\gfi(y)\phl(y)\dy>1-A_{\lambda,\text{i}}\geq0$ and $1-\int_{\lomega}\gfo(y)\psl(y)\dy>1-A_{\lambda,\text{o}}\geq0$. Along with \eqref{dd}, we arrive at
\begin{align*}
    \det(\mathcal{I}- \mathcal{R}_{\lambda})>(1-A_{\lambda,\text{i}})(1-A_{\lambda,\text{o}})-A_{\lambda,\text{i}}A_{\lambda,\text{o}}=1-A_{\lambda,\text{i}}-A_{\lambda,\text{o}}\geq0,
\end{align*}
implying that $\mathcal{I}- \mathcal{R}_{\lambda}$ is invertible. Along with Proposition~\ref{m-cor}(i), we obtain the uniqueness of the solution $\ue$ of \eqref{equl}--\eqref{bdu} when \eqref{c-g<1} holds. Therefore, the proof of Corollary~\ref{cor-ch} is complete.
\end{proof}

  Finally, let us consider the case that $\gfi$ and $\gfo$ are non-negative.  Applying the technique of the maximum principle, we have the following result.

\begin{corollary}\label{m-prop}
Under the same assumptions as in Proposition~\ref{m-cor}, we assume that both $\gfi$ and $\gfo$ are non-negative, and $\lambda>0$ satisfies
\begin{align}\label{g<1}
 \int_{\lomega}\gfi(y)\dy<1\qquad\,and\qquad   \int_{\lomega}\gfo(y)\dy<1.
\end{align}
{Then,  one of the following statements holds:
\begin{itemize}
    \item[(S1)] $\det(\mathcal{I}-\mathcal{R}_{\lambda})\neq0$ and \eqref{equl}--\eqref{bdu} has a unique solution;
    \item[(S2)] $\det(\mathcal{I}-\mathcal{R}_{\lambda})=0$ and \eqref{equl}--\eqref{bdu} has no solution.
\end{itemize}
 }
\end{corollary}
\begin{proof}
We fix $\lambda>0$ satisfying \eqref{g<1}. First, we prove that the equation~\eqref{equl}--\eqref{bdu} has at most one solution. Suppose by contradiction that \eqref{equl}--\eqref{bdu} has at least two distinct solutions $U_1$ and $U_2$. For $U:=U_1-U_2$, we have 
\begin{align}\label{eq>U}
-\Delta\,U(x)+\abar(x)\cdot\nabla\,U(x)+h(x)U(x)=0\quad\mathrm{in}\quad \lomega,
\end{align}
where $U$ is {subject to} the integral-type boundary conditions
\begin{align}\label{bd>U}
\begin{cases}
\ds\,U(x)= \int_{\lomega}\gfi(y)U(y)\dy,\quad\,x\in\gi:=\left\{\lambda z:z\in\partial\Omega_{\mathrm{in}}\right\},\vspace{3pt}\\
\ds\,U(x)= \int_{\lomega}\gfo(y)U(y)\dy,\quad\,x\in\go:=\left\{\lambda z:z\in\partial\Omega_{\mathrm{out}}\right\}.
\end{cases}
\end{align}
To obtain a contradiction, we consider two situations without loss of generality.\\
{\bf Case~1.} $U$ attains both its maximum and minimum values at the boundary points. Since $\gfi$ and $\gfo$ are both non-negative, by \eqref{g<1} and \eqref{bd>U}, one immediately obtains $\ds\max_{\lomega}U\leq0\leq\min_{\lomega}U$. Hence, $U\equiv0$, implying that $U_1=U_2$, a contradiction.\\
{\bf Case~2.} $U$ attains its maximum value at an interior point $x_{\lambda}$ and its minimum value on the boundary. Then, by \eqref{ab619} and \eqref{eq>U}, we have $\ds\max_{\lomega}U=U(x_{\lambda})\leq0$. On the other hand, $\ds\min_{\lomega}U\geq0$ holds trivially due to \eqref{g<1}. This also implies $U_1=U_2$ and {leads to}  a contradiction. 

Hence, by Cases 1--2, equation~\eqref{equl}--\eqref{bdu} has at most one  solution. Furthermore, if  $\lambda>0$ satisfies $\det(\mathcal{I}-\mathcal{R}_{\lambda})\neq0$, then by Proposition~\ref{m-cor}(i), equation~\eqref{equl}--\eqref{bdu} has a unique solution, i.e. (S1) holds. If  $\lambda>0$ satisfies $\det(\mathcal{I}-\mathcal{R}_{\lambda})=0$, then by Proposition~\ref{m-cor}(iii), there holds $\det\mathcal{C}_{\Psi_{\lambda}}\neq0$, and the equation~\eqref{equl}--\eqref{bdu} has no solution, i.e., (S2) holds. Therefore, the proof of Corollary~\ref{m-prop} is completed.
\end{proof}
\subsection{Proof of Theorem~\ref{m-thm}} 
First, we  show that there exists $\lambda_*>0$ such that $\det(\mathcal{I}- \mathcal{R}_{\lambda})>0$ for $\lambda\in(0,\lambda_*)$. Recall $\Omega:=\Omega_{\mathrm{out}}\setminus\overline{\Omega_{\mathrm{in}}}$. One can check that
\begin{equation}\label{int-gg}
    \begin{aligned}
\int_{\lomega}\left(|\gfi(y)|+|\gfo(y)|\right)\mathrm{d}y=&\,\lambda^N\int_{\Omega_{\mathrm{out}}\setminus\overline{\Omega_{\mathrm{in}}}}\left(|\gfi(\lambda{z})|+|\gfo(\lambda{z})|\right)\mathrm{d}z\\
\leq&\,\frac{1}{\left(\text{dist}(0,\partial\Omega_{\mathrm{in}})\right)^N}\int_{\Omega_{\mathrm{out}}\setminus\overline{\Omega_{\mathrm{in}}}}|\lambda{z}|^N\left(|\gfi(\lambda{z})|+|\gfo(\lambda{z})|\right)\mathrm{d}z.    
\end{aligned}
\end{equation}
Here we have verified $0<\text{dist}(0,\partial\Omega_{\mathrm{in}})\leq\min\limits_{z\in\,\overline{\Omega_{\mathrm{out}}}\setminus\Omega_{\mathrm{in}}}|z|$ since $0\in\Omega_{\mathrm{in}}$. Note also that $\lim\limits_{\lambda\to0+}\max\limits_{z\in\,\overline{\Omega_{\mathrm{out}}}\setminus\Omega_{\mathrm{in}}}|\lambda{z}|=0$ since $\Omega_{\mathrm{out}}\setminus\overline{\Omega_{\mathrm{in}}}$ is a bounded domain. Thus, by \eqref{gio} and \eqref{int-gg}, we have 
\begin{equation*}
    \lim\limits_{\lambda\to0+}\int_{\lomega}\left(|\gfi(y)|+|\gfo(y)|\right)\mathrm{d}y=0,
\end{equation*}
 asserting that there exists $\lambda_*>0$ such that \eqref{c-g<1} holds for $\lambda\in(0,\lambda_*)$. Then, by Corollary~\ref{cor-ch}, we obtain $\det(\mathcal{I}- \mathcal{R}_{\lambda})>0$ for $\lambda\in(0,\lambda_*)$.

On the other hand, by \eqref{ph-*}--\eqref{ps-*}, one may check that, as $\lambda\gg1$, 

\begin{equation}\label{feb-0}
   \begin{aligned}
  \left|\int_{\lomega}\gfi(y)\phl(y)\dy\right|
    \leq&\max_{\overline{\lomega}}|\gfi|\left(\int_{\lomega}\exp\left(-M^*\lambda^{-\frac{\kappa}{2}}\mathrm{dist}(y,{\partial\lomega})\right)\dy\right)\\
   \leq&\,\lambda^N\max_{z\in\overline{\Omega}}|\gfi(\lambda{z})|\left(\int_{\Omega}\exp\left(-M^*\lambda^{1-\frac{\kappa}{2}}\mathrm{dist}(z,{\partial\Omega})\right)\dz\right).
\end{aligned} 
\end{equation}
Furthermore, we carefully deal with the last term of \eqref{feb-0}. In order to achieve a more refined estimate, we set a constant $\tau\in(0,1-\frac{\kappa}{2})$ independent of $\lambda$ such that as $\lambda>0$ is sufficiently large, the subdomain
\begin{equation*}
    \Omega_{\lambda^{-\tau}}^*:=\{z\in\Omega:\,\text{dist}(z,\partial\Omega)>\lambda^{-\tau}\}
\end{equation*}
 is nonempty, and each interior point of $\Omega\setminus\overline{\Omega_{\lambda^{-\tau}}^*}$ has a unique nearest point on $\partial\Omega$. Note that  $\partial\Omega_{\lambda^{-\tau}}^*$ is smooth and  $\max\limits_{z\in\overline{\Omega_{\lambda^{-\tau}}^*}}\exp(-M^*\lambda^{1-\frac{\kappa}{2}}\mathrm{dist}(z,{\partial\Omega}))\leq\exp\left(-M^*\lambda^{1-\frac{\kappa}{2}-\tau}\right)$. Utilizing the coarea formula~(cf.~\cite[Theorems~3.10 and 3.14]{EG2015}), one may check that, as $\lambda\gg1$,
\begin{equation}\label{feb}
   \begin{aligned}
    \int_{\Omega}\exp&\left(-M^*\lambda^{1-\frac{\kappa}{2}}\mathrm{dist}(z,{\partial\Omega})\right)\dz\\
    =&\,\int_0^{\lambda^{-\tau}}\exp\left(-M^*\lambda^{1-\frac{\kappa}{2}}t\right)\mathcal{H}^{N-1}(\{z\in\Omega:\,\text{dist}(z,\partial\Omega_{\mathrm{in}})=t\})\,\text{d}t\\
    &\quad+\int_0^{\lambda^{-\tau}}\exp\left(-M^*\lambda^{1-\frac{\kappa}{2}}t\right)\mathcal{H}^{N-1}(\{z\in\Omega:\,\text{dist}(z,\partial\Omega_{\mathrm{out}})=t\})\,\text{d}t\\
    &\quad+\int_{\Omega_{\lambda^{-\tau}}^*}\exp\left(-M^*\lambda^{1-\frac{\kappa}{2}}\mathrm{dist}(z,{\partial\Omega})\right)\dz\\
   \leq&\,2\left(\mathcal{H}^{N-1}(\partial\Omega_{\mathrm{in}})+\mathcal{H}^{N-1}(\partial\Omega_{\mathrm{out}})\right)\int_0^{\lambda^{-\tau}}\exp\left(-M^*\lambda^{1-\frac{\kappa}{2}}t\right)\mathrm{d}t\\
   &\quad+\mathcal{H}^{N}(\Omega)\exp\left(-M^*\lambda^{1-\frac{\kappa}{2}-\tau}\right)\\
  \leq&\,\frac{2\lambda^{\frac{\kappa}{2}-1}}{M^*}\left(\mathcal{H}^{N-1}(\partial\Omega_{\mathrm{in}})+\mathcal{H}^{N-1}(\partial\Omega_{\mathrm{out}})\right)+\mathcal{H}^{N}(\Omega)\exp\left(-M^*\lambda^{1-\frac{\kappa}{2}-\tau}\right)\leq{C_3}\lambda^{\frac{\kappa}{2}-1},
\end{aligned} 
\end{equation}
where $\mathcal{H}^{d}$ represents the $d$-dimensional Hausdorff measure on $\mathbb{R}^N$, and $C_3$ is a positive constant independent of $\lambda$. Since $0<\lambda^{-\tau}\ll1$, in \eqref{feb} we have used $\exp\left(-M^*\lambda^{1-\frac{\kappa}{2}-\tau}\right)\ll\lambda^{\frac{\kappa}{2}-1}$ (as $\lambda\gg1$), and the fact that $\mathcal{H}^{N-1}(\{z\in\Omega:\,\text{dist}(z,\partial\Omega_{\mathrm{in}})=t\})\leq2\mathcal{H}^{N-1}(\partial\Omega_{\mathrm{in}})$ and $\mathcal{H}^{N-1}(\{z\in\Omega:\,\text{dist}(z,\partial\Omega_{\mathrm{out}})=t\})\leq2\mathcal{H}^{N-1}(\partial\Omega_{\mathrm{out}})$ as $0<t\leq\lambda^{-\tau}\ll1$. Hence, by \eqref{gio} and \eqref{feb-0}--\eqref{feb}, one arrives at
\begin{align}\label{g12}
   \left|\int_{\lomega}\gfi(y)\phl(y)\dy\right|\leq{C_3}\lambda^{N+\frac{\kappa-2}{2}}\max_{z\in\overline{\Omega}}|\gfi(\lambda z)|\xrightarrow{\lambda\to\infty}0.
\end{align}
Similarly, we can obtain
\begin{align}\label{g3}
   \int_{\lomega}\gfo(y)\psl(y)\dy,\,\int_{\lomega}\gfi(y)\psl(y)\dy,\,\int_{\lomega}\gfo(y)\phl(y)\dy\xrightarrow{\lambda\to\infty}0.
\end{align}
As a consequence, by \eqref{big-i-a}--\eqref{big-ii}, \eqref{dd} and \eqref{g12}--\eqref{g3}, we have
\begin{align}\label{g5}
  \det(\mathcal{I}- \mathcal{R}_{\lambda})\xrightarrow{\lambda\to\infty}  1,\quad\det\mathcal{C}_{\psl}\xrightarrow{\lambda\to\infty}b_{\mathrm{i}},\quad\det\mathcal{C}_{\phl}\xrightarrow{\lambda\to\infty}b_{\mathrm{o}}.
\end{align}
Therefore, by \eqref{Bio} and \eqref{g5}, there exists a sufficiently large number $\lambda^*$ such that $\det(\mathcal{I}- \mathcal{R}_{\lambda})>0$ for $\lambda\in(\lambda^*,\infty)$. Consequently,
\begin{align}\label{det}
 \det(\mathcal{I}- \mathcal{R}_{\lambda})>0\qquad\mathrm{as}\quad\lambda\in(0,\lambda_*)\cup(\lambda^*,\infty).   
\end{align}
By \eqref{det} and Proposition~\ref{m-cor}(i), we obtain the uniqueness of solutions to \eqref{equl}--\eqref{bdu} with $\lambda\in(0,\lambda_*)\cup(\lambda^*,\infty)$.

It remains to prove \eqref{u-in-e}. By \eqref{B-io} and \eqref{g5}, we have $|B_{\lambda,\mathrm{i}}|+|B_{\lambda,\mathrm{o}}|\leq2(b_{\mathrm{i}}+b_{\mathrm{o}})$ for sufficiently large $\lambda>0$.
On the other hand, as $\lambda\gg1$, by \eqref{sol-u} and \eqref{ph-*}--\eqref{ps-*}, we arrive at an interior estimate of $\ue$:
\begin{align*}
 |\ue(x)|\leq |B_{\lambda,\mathrm{i}}|  \phl(x)+|B_{\lambda,\mathrm{o}}|  \psl(x)\leq(|B_{\lambda,\mathrm{i}}|+|B_{\lambda,\mathrm{o}}|)\exp\left(-M^*\lambda^{-\frac{\kappa}{2}}\mathrm{dist}(x,{\partial\lomega})\right),\quad\forall\,x\in   \overline{\lomega}.
\end{align*}
Therefore, we obtain \eqref{u-in-e} with $C^*=2(b_{\mathrm{i}}+b_{\mathrm{o}})$ and complete the proof of Theorem~\ref{m-thm}.

\begin{figure}[ht]
\includegraphics[width=6.6cm]{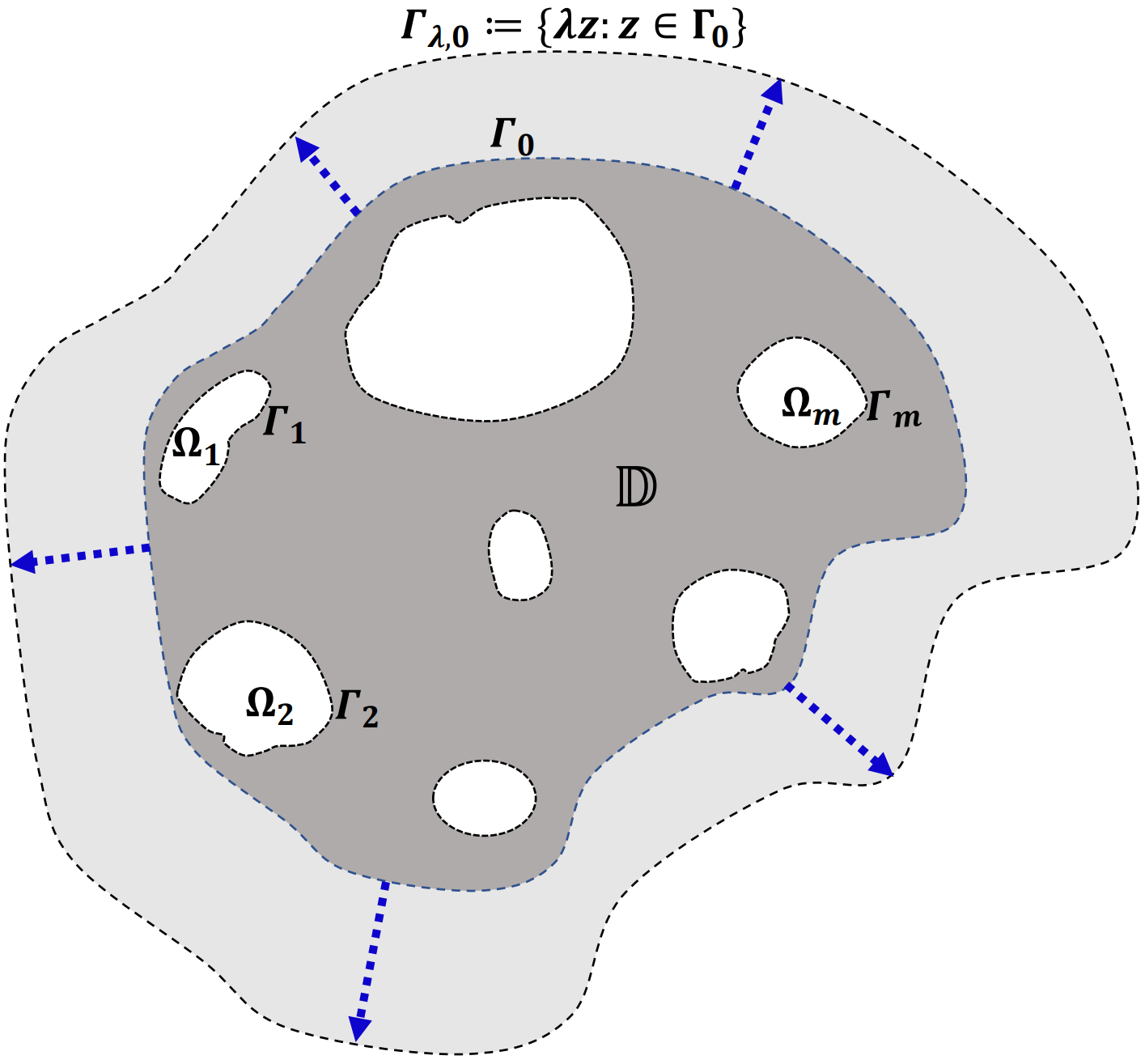}
\caption{The dark grey region shows a domain $\mathbb{D}:=\Omega_{0}\setminus\bigcup_{1\leq k\leq m}\overline{\Omega_{k}}$ with many boundary components~$\Gamma_j$'s, where $\Gamma_j:=\partial\Omega_j$, $j=0,1,...,m$. Here, $\Omega_0$ and $\Omega_{k}\Subset\Omega_{0}$ are bounded smooth domains in $\mathbb{R}^N$, $N\geq2$, and    all $\overline{\Omega_{k}}$'s are disjoint to each other.   Furthermore, we assume $0\notin\overline{\mathbb{D}}$. The corresponding expanding domain is defined by $\mathbb{D}_{\lambda}:=\{\lambda z:z\in\mathbb{D}\}$ with $m+1$ boundary components~$\Gamma_{\lambda,j}$'s.\label{fig-m}}
\end{figure}

\section{\bf A uniqueness result in the expanding domain with many boundary components}\label{sec-mb}
 Let $\mathbb{D}_{\lambda}$ be a bounded domain with $m+1$ smooth boundary components $\Gamma_j$'s, $j=0,1,...,m$. For the definitions of $\mathbb{D}_{\lambda}$ and $\Gamma_{\lambda,j}$'s, we refer the readers to the caption of Figure~\ref{fig-m}. Note that $\abar$ and $h$ are well-defined in $\overline{\mathbb{D}_{\lambda}}$ since $0\not\in\overline{\mathbb{D}_{\lambda}}$. It suffices to consider the equation
\begin{align}\label{eqv}
-\Delta\ve(x)+\abar(x)\cdot\nabla\ve(x)+h(x)\ve(x)=0\quad\mathrm{in}\quad \mathbb{D}_{\lambda},
\end{align}
with the integral-type boundary conditions
\begin{equation}\label{bdv}
    \ve(x)=b_j+  \int_{\mathbb{D}_{\lambda}}g_j(y)\ve(y)\dy,\quad\,\,x\in\Gamma_{\lambda,j}:=\{\lambda z:z\in\Gamma_{j}\},\quad\,j=0,1,...,m,
\end{equation}
where $b_j\in\mathbb{R}$ and $g_j\in\mathrm{C}(\mathbb{R}^N\setminus\{0\};\mathbb{R})$ are independent of $\lambda$. Here, we will focus primarily on the uniqueness problem of \eqref{eqv}--\eqref{bdv}.

Because there are $m+1$ boundary components $\Gamma_{\lambda,j}$'s, for each $\eta\in\{0,1,...,m\}$, let $\phi_{\lambda,\eta}$ be the unique solution of the equation
\begin{align}\label{eqphj}
\begin{cases}
-\Delta\phi_{\lambda,\eta}(x)+\abar(x)\cdot\nabla\phi_{\lambda,\eta}(x)+h(x)\phi_{\lambda,\eta}(x)=0\quad\mathrm{in}\quad \mathbb{D}_{\lambda},\\
\phi_{\lambda,\eta}(x)=1,\,\,\,x\in\Gamma_{\lambda,\eta},\qquad\phi_{\lambda,\eta}(x)=0,\,\,\,x\in\Gamma_{\lambda,j}\,\,\,(j\neq\eta).
\end{cases}
\end{align}
Owing to \eqref{ab619}, $\phi_{\lambda,\eta_1}$ and $\phi_{\lambda,\eta_2}$ are linearly independent ($\eta_1\neq\eta_2$). Hence, each solution $\ve$ of \eqref{eqv}--\eqref{bdv} (if it exists) is a linear combination of all $\phi_{\lambda,\eta}$'s with the coefficients determined by \eqref{bdv}. Therefore, we can apply a similar argument as in the proofs of Proposition~\ref{m-cor} and Theorem~\ref{m-thm} to arrive at the following result.

\begin{theorem}\label{thm2}
Assume that $b_j$'s, $j=0,1,...,m$, are constants independent of $\lambda$,  $h\in\mathrm{C}^{\alpha}(\mathbb{R}^N;\mathbb{R})$ and $\abar\in\mathrm{C}^{\alpha}(\mathbb{R}^N;\mathbb{R}^N)$ satisfying \eqref{ab619}. Let $g_j\in\mathrm{C}(\mathbb{R}^N;\mathbb{R})$ satisfy \eqref{gio} with $g=g_j$. Then, there hold the followings:
\begin{itemize}
    \item[(i)] Equation~\eqref{eqv}--\eqref{bdv} has a unique solution if and only if
 $\lambda>0$ satisfies
 \begin{equation*}
     \det\left(\mathcal{I}_{m+1}- \mathcal{G}_{\lambda}\right)\neq0,
 \end{equation*}
  where $\mathcal{I}_{m+1}$ is the $(m+1)\times(m+1)$ identity matrix, and
\begin{align*}
   \mathcal{G}_{\lambda}=   \left[\mathsf{G}_{\lambda,ij}\right]_{(m+1)\times(m+1)}\quad\text{with}\quad\mathsf{G}_{\lambda,ij}=\int_{\mathbb{D}_{\lambda}}g_{i-1}(y)\phi_{\lambda,j-1}(y)\dy,\,\,i,j\in\{1,...,m+1\}.
\end{align*}
\item[(ii)] There exist $\widetilde{\lambda}_*$ and $\widetilde{\lambda}^*$ with $0<\widetilde{\lambda}_*<\widetilde{\lambda}^*$ such that for $\lambda\in(0,\widetilde{\lambda}_*)\cup(\widetilde{\lambda}^*,\infty)$,  \eqref{eqv}--\eqref{bdv} has a unique solution~$\ve$ satisfying
  \begin{align*}
\ve(x)=\sum_{\eta=0}^mc_{\eta}\phi_{\lambda,\eta}(x),
  \end{align*}
where $\phi_{\lambda,\eta}$ is the unique solution of \eqref{eqphj},
and 
\begin{align*}
\begin{bmatrix}
c_0 \vspace{3pt}\\
c_1 \vspace{3pt}\\
\vdots\\
c_m \vspace{3pt}
\end{bmatrix}
  =  \left(\mathcal{I}_{m+1}- \mathcal{G}_{\lambda}\right)^{-1} \begin{bmatrix}
b_0 \vspace{3pt}\\
b_1 \vspace{3pt}\\
\vdots\\
b_m \vspace{3pt}
\end{bmatrix}.
\end{align*} 
\end{itemize}
\end{theorem}
\begin{proof}
The proof of (i) follows the same argument as in the proof of Proposition~\ref{m-cor}(i). We omit the detail here because of its simplicity.

Now we state the proof of (ii). Applying a similar argument as \eqref{ph-*} to \eqref{eqphj}, we can establish the following estimate for each $\phi_{\lambda,\eta}$:
 \begin{align}\label{ph-ls}
 0\leq\phi_{\lambda,\eta}(x)\leq \exp\left(-M^*\lambda^{-\frac{\kappa}{2}}\mathrm{dist}(x,{\partial\mathbb{D}_{\lambda}})\right),\quad\mathrm{for\,\,all}\,\,x\in   \overline{\mathbb{D}_{\lambda}}. 
 \end{align}
 Moreover, by \eqref{gio} with $g=g_{i-1}$ and  \eqref{ph-ls} with $\eta=j-1$, we can use the same arguments as \eqref{feb-0}--\eqref{feb} to obtain $\lim\limits_{\lambda\to\infty}\mathsf{G}_{\lambda,ij}=0$. As a consequence, we arrive at $\lim\limits_{\lambda\to\infty}\det\left(\mathcal{I}_{m+1}- \mathcal{G}_{\lambda}\right)=1$. On the other hand, since $0\leq\phi_{\lambda,\eta}\leq1$ in $\overline{\mathbb{D}_{\lambda}}$ and $g_j$ satisfies \eqref{gio} with $g=g_j$ near $z=0$, one can follow the same argument as \eqref{int-gg} to obtain $\lim\limits_{\lambda\to0+}\mathsf{G}_{\lambda,ij}=0$ which implies $\lim\limits_{\lambda\to0+}\det(\mathcal{I}_{m+1}- \mathcal{R}_{\lambda})=1$.

Consequently, there exist $\widetilde{\lambda}_*$ and $\widetilde{\lambda}^*$ with $0<\widetilde{\lambda}_*<\widetilde{\lambda}^*$ such that for $\lambda\in(0,\widetilde{\lambda}_*)\cup(\widetilde{\lambda}^*,\infty)$, $\det(\mathcal{I}_{m+1}- \mathcal{R}_{\lambda})\neq0$ holds. Along with (i), we obtain the uniqueness result of \eqref{eqv}--\eqref{bdv} and complete the proof of Theorem~\ref{thm2}. 
\end{proof}

\section{\bf Concluding remarks  and future plans}\label{sec-open}
For equation~\eqref{equl}--\eqref{bdu} in the domain~$\lomega$ defined by \eqref{domain}, let us consider the set
\begin{align*}
    \mathfrak{S}:=\{\lambda>0:\det(\mathcal{I}- \mathcal{R}_{\lambda})=0\},
\end{align*}
where $\mathcal{I}$ and $\mathcal{R}_{\lambda}$ are defined by \eqref{big-i-a}. Generally, $\mathfrak{S}$ is non-empty because Example~\ref{ex1} is the case. Proposition~\ref{m-cor}   provide a basic understanding that equation~\eqref{equl}--\eqref{bdu} has a unique solution (satisfying \eqref{sol-u}) if and only if $\lambda\in\mathbb{R}^{>0}\setminus\mathfrak{S}$. Furthermore,  Theorem~\ref{m-thm} shows that under assumptions~\eqref{ab619} and \eqref{gio},  $\mathfrak{S}$ is contained in a bounded interval~$[\lambda_*,\lambda^*]$. For the general case when the domain $\lomega$ is replaced with a bounded smooth domain $\mathbb{D}_{\lambda}$ with many boundary components, the corresponding uniqueness result is stated by Theorem~\ref{thm2}. To the best of our knowledge, these results may contribute to the first understanding of the roles of domain size on the existence and uniqueness of the equation~\eqref{equl}--\eqref{bdu}.

{We shall emphasise that for $\lambda_s\in\mathfrak{S}$,  \eqref{equl}--\eqref{bdu} has infinitely many solutions if $\det\mathcal{C}_{\psi_{\lambda_s}}=0$, and has no solution if $\det\mathcal{C}_{\psi_{\lambda_s}}\neq0$}  (cf. Proposition~\ref{m-cor}(ii)--(iii)). Therefore, examining the number of elements in $\mathfrak{S}$ plays a crucial role in the uniqueness of  equation~\eqref{equl}--\eqref{bdu}.  Under assumptions~\eqref{ab619} and \eqref{gio}, we conjecture that $\mathfrak{S}$ only has finitely many elements since Example~\ref{ex1} provides  relevant evidence on this conjecture. The rigorous proof is a great challenge.

On the other hand, even if we  consider only the annular domain in Example~\ref{ex1}, a closer observation of equation \eqref{alg-eq} shows that $\lambda^*$ is influenced by the domain geometry. Thus, it is expected that for the  domain~$\Omega:=\Omega_{\mathrm{out}}\setminus\overline{\Omega_{\mathrm{in}}}$ with {non-constant} boundary mean curvature, the roots of equation $\det(\mathcal{I}- \mathcal{R}_{\lambda})=0$ may depend on the geometries of $\Omega_{\mathrm{in}}$ and $\Omega_{\mathrm{out}}$. A question naturally arises at this point:
\begin{align*}
    \text{`}&\textit{How\,\,to\,\,identify\,\,the\,\,effect\,\,of\,\,the\,\,geometries\,\,of}\,\, \Omega_{\mathrm{in}}\,\,\text{and}\,\,\Omega_{\mathrm{out}}\,\,
\textit{on\,\,the\,\,roots\,\,of}\,\, \det(\mathcal{I}- \mathcal{R}_{\lambda})=0\\
&\,\textit{more\,\,precisely?'}
\end{align*}

Furthermore, as shown in  Examples~\ref{ex1} and \ref{ex2}, it is expected that, for an equation in a bounded domain with the integral-type boundary conditions, the relationship between the domain capacity and the boundary condition may affect the uniqueness result of the solutions, which will be our future research directions.

\section{\bf Appendix: Proof of \eqref{apid}}
In this section, we state the proof of \eqref{apid} for the sake of clarity. By \eqref{moon-g} with $C_0 =\ds \frac{5\text{e}}{2\pi(2\sin 1+\cos 1)}$ and  \eqref{moon-d},  one may check via simple calculations that 
\begin{align*}
 \det(\mathcal{I}- \mathcal{R}_{\lambda})=&\,-\frac{\text{e}^{-2\lambda+1}}{\text{e}^{-1}-\text{e}^{-2\lambda+1}}\left[1- \frac{5}{2\sin1+\cos1}\left(\cos 1 -\cos \lambda +\frac15 (2\sin\lambda+\cos\lambda)\right)\right]\\
 =&\,-\frac{2\text{e}^{-2\lambda+2}(\sin1-2\cos1-\sin\lambda+2\cos\lambda)}{(1-\text{e}^{-2\lambda+2})(2\sin1+\cos1)}\\
 =&\,-\frac{4\sqrt{5}\text{e}^{-2\lambda+2}}{(1-\text{e}^{-2\lambda+2})(2\sin1+\cos1)}\cos\frac{1+\lambda-2\theta_0}{2}\sin\frac{1-\lambda}{2},\quad\text{where}\,\,\theta_0=\arcsin\frac{2}{\sqrt{5}}.
\end{align*}
 Here we employed some trigonometric identities to obtain $(\sin1-2\cos1)-(\sin\lambda-2\cos\lambda)=\sqrt{5}\big(\sin(1-\theta_0)-\sin(\lambda-\theta_0)\big)=2\sqrt{5}\cos\frac{1+\lambda-2\theta_0}{2}\sin\frac{1-\lambda}{2}$.  Note that $\lambda>1$ and $\frac{4\sqrt{5}\text{e}^{-2\lambda+2}}{(1-\text{e}^{-2\lambda+2})(2\sin1+\cos1)}\neq0$. Thus, we have
\begin{equation*}
 \det(\mathcal{I}- \mathcal{R}_{\lambda})=0\,\Longleftrightarrow\,\,\cos\frac{1+\lambda-2\theta_0}{2}\sin\frac{1-\lambda}{2}=0,
\end{equation*}
which immediately implies \eqref{apid}, as desired.

\subsection*{Acknowledgement} The authors would like to thank the anonymous referee and the editor who pointed out the importance of the domain capacity and made insightful suggestions that contributed to enhancing the overall quality of the original manuscript. The research of C.-C. Lee was partially supported by MOST grants with numbers 108-2115-M-007-006-MY2 and 110-2115-M-007-003-MY2 of Taiwan. The research of M. Mizuno was partially supported by JSPS KAKENHI grants with numbers JP18K13446 and JP22K03376.

\end{document}